\documentclass{amsproc}

\usepackage{tikz}
\usetikzlibrary{automata,cd,shapes}
\usepackage{faktor} 

\usepackage{amssymb}
\usepackage{mathtools}

\usepackage{}

\newtheorem{theorem}{Theorem}

\newtheorem{lemma}{Lemma}[section]

\theoremstyle{definition}
\newtheorem{definition}[lemma]{Definition}
\newtheorem{notation}[lemma]{Notation}

\theoremstyle{remark}
\newtheorem{remark}[lemma]{Remark}

\numberwithin{equation}{section}
\DeclareMathOperator{\Aut}{Aut}
\DeclareMathOperator{\Sym}{Sym}
\DeclareMathOperator{\supp}{supp}

\newcommand\N{\mathbb{N}}
\newcommand\R{\mathbb{R}}

\newcommand\T{\mathcal{T}}

\newcommand\D{\mathcal{D}}
\newcommand\G{\mathcal{G}}
\newcommand\HH{\mathcal{H}} 
\newcommand\Net{\mathcal{N}}

\newcommand\Id{\mathbf{1}} 
\newcommand\charr{\chi}

\newcommand{\Wr}[3]{#1 \operatorname{Wr}_{#3} #2}
\newcommand{\wwr}[3]{#1 \wr_{#3} #2}
\newcommand{\aut}{\Aut_{\operatorname{fin.}}}
\newcommand{\autB}{\Aut_{\mathcal{B}}}
\newcommand{\autF}{\Aut_{\mathcal{F}}}
\newcommand{\autBfs}{\Aut^{f.s.}_{\mathcal{B}}}
\newcommand{\norm}[1]{\left\lVert#1\right\rVert}
\newcommand{\card}[1]{\left\lvert#1\right\rvert}
\newcommand{\qedcited}{\hfill \ensuremath{\triangle}}

\begin{document}

\title{Amenability of Bounded Automata Groups on Infinite Alphabets}

\author{Bernhard Reinke}
\address{Institut de Mathématiques (UMR CNRS7373) \\
Campus de Luminy \\
163 avenue de Luminy --- Case 907 \\
13288 Marseille 9 \\
France}
\curraddr{}
\email{}
\thanks{}


\date{}

\begin{abstract}
  We study the action of groups generated by bounded activity automata
	with infinite alphabets on their orbital Schreier graphs.
  We introduce an amenability criterion for such groups based on the
	recurrence of the first level action. This criterion is a natural
	extension of the result that all groups generated by bounded activity
	automata with finite alphabets are amenable. Our motivation comes from
	the investigation of iterated monodromy groups of entire functions.
\end{abstract}

\maketitle

\section{Introduction}
Self-similar groups provide many examples of ``exotic'' amenable groups.
The Grigorchuk group~\cite{Grigorchuk1983} was the first example of
a group of intermediate growth. Groups of intermediate growth are always
amenable, but not elementary amenable (see~\cite{chou1980}).
The basilica group is amenable~\cite{bartholdi2005amenability}, but not
elementary subexponentially amenable~\cite{GZ2002}.

Both the Grigorchuk group and the basilica group are examples of
automata groups on a two-letter alphabet of bounded activity growth. They fit into the hierarchy
of polynomial activity growth 
introduced in~\cite{sidki2000}, where both finite and infinite alphabets are considered.
Under certain assumptions (which are always satisfied for finite alphabets), these groups do not contain free subgroups (see~\cite{sidki2004finite}).
For finite alphabets, it is shown in~\cite{Bartholdi2010} that the group generated by bounded activity automata is amenable. A large family of such groups
are iterated monodromy groups of post-critically finite
polynomials~\cite{Nekrashevych2009}. Furthermore, in~\cite{AAV} it is 
shown that automata groups on finite of linear activity growths are
amenable. The techniques of~\cite{Bartholdi2010} and~\cite{AAV} have
been conceptualized in~\cite{juschenko2016}.

In~\cite{Reinkeentire}, we show that iterated monodromy groups of post-singularly finite entire functions are given by bounded activity automata on infinite alphabets.
We expect many similarities of these groups to their polynomial counterparts, so one question in particular is amenability. We can not expect all iterated monodromy groups of
post-singularly finite entire functions to be amenable, as there are entire functions with monodromy group $C_2 * C_2 * C_2$. 

In the forthcoming paper~\cite{Reinkeentire}, we show the following:
\begin{theorem}[Main application]
  Let $f$ be a post-singularly finite entire function. Then the iterated monodromy group of $f$ is amenable if and only if the monodromy group of $f$ is amenable.
	\label{thm:motivation}
\end{theorem}

In this paper we provide the main group theoretic part of the proof of this theorem. We 
show the following:

\begin{theorem}
  Let $P$ be an amenable subgroup of $\Sym(X)$. Suppose that the action of $P$ on $X$ is recurrent.
  Then $\autBfs(X^*;P)$ is amenable.
	\label{thm:maintheoremintro}
\end{theorem}
See Section~2 for a precise definition of $\autBfs(X^*;P)$, it is
roughly the groups of bounded activity automata were every first level
action is in $P$. 

We note that Theorem~\ref{thm:motivation} is our main motivation for Theorem~\ref{thm:maintheoremintro}, but this paper does not logically depend on~\cite{Reinkeentire}.

In Section~2, we start by introducing self-similar groups on infinite
alphabets and related concepts, such as the space of ends. We continue in Section~3 with a
discussion of recurrent random walks and how to pass from a recurrent
action on the alphabet to a recurrent action of a bounded activity group
on the space of ends. This will be a key ingredient to invoke the amenability criterion
of~\cite{juschenko2016} in Section~4 to prove
Theorem~\ref{thm:maintheoremintro}. In Section~5, we briefly discuss the forthcoming paper and further related open questions.

\emph{Acknowledgements.} We gratefully acknowledge support by the Advanced Grant HOLOGRAM by the European Research Council. Part of this research was
done during visits at Texas A\&M University and at UCLA\@.
We would like to thank our hosts, Volodymyr Nekrashevych and Mario Bonk, as well as the HOLOGRAM team, in particular Kostiantyn Drach, Dzmitry Dudko, Mikhail Hlushchanka, David Pfrang and Dierk Schleicher, for helpful discussions and comments.

\section{Regular trees}
In this section we introduce of self-similar groups
and other relevant concepts and fix the notation. 

\begin{definition}
  Let $X$ be a countable infinite set.
  The standard $X$-regular tree has as vertex set $X^*$, the set of finite
  words in $X$. Its root is the empty word $\emptyset$. Its edges are all pairs
  $(v, vx)$ for $v \in X^*, x \in X$. By abuse of notion, we denote the standard $X$-regular tree also as $X^*$, and
  we denote by $\Aut(X^*)$ the group of
  rooted tree automorphisms of $X^*$. We denote the identity of $\Aut(X^*)$ by $\Id$.

  For $v\in X^*$, let $vX^*$ be the subtree of all descendants of $v$. If $g \in
  \Aut(X^*), v \in X^*$, there is a unique $g_{|v} \in \Aut(X^*)$ given
  by $g(vw) = g(v)g_{|v}(w)$. This is called the \emph{section} of $g$ along $v$.

\end{definition}

  A set $S \subset \Aut(X^*)$ is called \emph{self-similar} if it is closed under
  taking sections, i.e.\ $g_{|v} \in S$ for all $g \in S, v \in X^*$. We
  are mainly interested in self-similar groups, i.e.\ subgroups $G \subset
  \Aut(X^*)$ that are self-similar as sets.

  For $g \in \Aut(X^*)$, the activity $\alpha_n(g) \in \N \cup \infty$ of $g$
  on level $n$ is the
  number of words $v$ of length $n$ for which the section $g|_v$ is not
  trivial. We denote by $\aut(X^*)$ the set of automorphisms with finite activity
  on every level. If $g \in \aut(X^*)$ has a $n$ so that $g_{|v} = \Id$ for all
  $v \in X^n$, we say that that $g$ is finitary. If $g \in \aut(X^*)$ has a $c
  \in \N$
  so that $\alpha_n(g) \leq c$ for all $n$, we say that $g$ has bounded activity.

  We denote by $\autB(X^*)$ the set of automorphisms with bounded activity, and
  by $\autF(X^*)$ the set of finitary automorphisms.

  We also have maps $\rho_n \colon \Aut(X^*) \rightarrow \Sym(X^n)$, which are
  induced by the action of $\Aut(X^*)$ on the $n$-th level. Let $P$ be a subgroup
  of $\Sym(X)$. Let $\Aut(X^*;P)$ denote the set of automorphisms such that
  $\rho_1(g_{|v}) \in P$ for all $v \in X^*$. We denote by $\aut(X^*;P),
  \autB(X^*;P),\autF(X^*;P)$ the intersections of $\aut(X^*), \autB(X^*), \autF(X^*)$
  with $\Aut(X^*;P)$ respectively.

Since we consider infinite alphabets, let us fix notations for the two versions of
wreath products.

\begin{notation}
  Let $A$ and $B$ be groups, $L$ be a set with an $A$-left action. The unrestricted
  wreath product $\left(\prod_{l \in L} B \right) \rtimes A$ is denoted $\Wr{B}{A}{L}$, the
  restricted wreath product $\left(\bigoplus_{l \in L} B\right) \rtimes A$ is denoted $\wwr{B}{A}{L}$.

  We will mainly work with the restricted wreath product. We denote the right factor embedding
  $A \rightarrow \wwr{B}{A}{L}$ by $\iota$, and by $b @ l$ the image of $b$ under the embedding
  of $B$ into the component indexed by $l$. 

\end{notation}
For a subgroup $P$ of $\Sym(X)$, we denote the $n$-th iterated restricted wreath product (along $X$) by $P_n$.
So $P_1 = P$ and $P_{n+1} = \wwr{P_n}{P}{X}$. Note that if $P$ is amenable, then all $P_n$ are amenable.
With this in mind we have the following:

\begin{lemma}
  \begin{eqnarray*}\label{eqn:wreathrecursion}
    \Aut(X^*;P) &\rightarrow& \Wr{\Aut(X^*;P)}{P}{X} \\
    g &\mapsto& (x \mapsto g_{|x}, \rho_1(g))
  \end{eqnarray*}
  is an isomorphism of groups. It restricts to isomorphisms
  \begin{eqnarray*}
    \aut(X^*;P) &\cong& \wwr{\aut(X^*;P)}{P}{X} \\
    \autB(X^*;P)&\cong& \wwr{\autB(X^*;P)}{P}{X} \\
    \autF(X^*;P)&\cong& \wwr{\autF(X^*;P)}{P}{X} \\
  \end{eqnarray*}
  \qedcited
\end{lemma}
For the first line, see for example~\cite{sidki2000}. By iteration, we also get isomorphisms
  \begin{eqnarray*}\label{eqn:wreathrecursionn}
    \aut(X^*;P) &\rightarrow& \wwr{\aut(X^*;P)}{P_n}{X} \\
    g &\mapsto& (v \mapsto g|v, \rho_n(g))
  \end{eqnarray*}
  and $\aut(X^*;P) \cong \aut({X^n}^*;P_n)$. 
\subsection{Action on space of ends $X^\omega$}
We will also use the action of $\Aut(X^*)$ on the space of ends of $X^*$.
The set of ends of $X^*$ can be identified with $X^\omega$, the set of right infinite words in $X$.
The open cylinder sets $C(v) = \left\{ vw \colon w \in X^\omega\right\}$ form a basis of the end topology on $X^\omega$. Since $X$ is countable infinite,
$X^\omega$ is homeomorphic to the Baire space $\N^\N$, in particular $X^\omega$ is Hausdorff, but not locally compact.
The action of $\Aut(X^*)$ on $X^\omega$ is faithful, so we can also think of elements of $\Aut(X^*)$ as homeomorphisms on $X^\omega$.
We will use the language of germs: these are equivalence classes of pairs $(g, w) \in \Aut(X^*) \times X^\omega$, where
$(g,w) \sim (h, w')$ if $w = w'$ and $g$ and $h$ agree on a neighborhood of $w$. Since we only consider germs of $\Aut(X^*)$,
and $\left\{C(v) \colon v \text{ is a prefix of } w\right\}$ forms a neighborhood basis, $(g,w) \sim (h, w')$ is equivalent to $w=w',g(w) = h(w)$ and
$g_{|v} = h_{|v}$ for some $v$ prefix of $w$. We denote by $\T$ the groupoid of germs of tail equivalences, that is germs of the form $(g, w)$ with $g_{|v}$ trivial for some prefix $v$ of $w$.
Given a groupoid of germs $\HH$, we denote by $\left[\left[\HH\right]\right]$ the set of global homeomorphisms, such that all their germs belong to $\HH$.

We have $\autF(X^*) \subset \left[ \left[ \T \right] \right]$. In contrast to the case when $X$ is finite, we do not have equality, as we can easily produce elements in
$\left[ \left[ \T \right] \right]$ which are not even in $\aut(X^*)$.

If $w, w' \in X^\omega$ can be factored as $w=vu, w'=v'u$ with $v, v' \in X^n,u \in X^\omega$, we say that $w$ and $w'$ are $n$-\emph{tail equivalent}.
We say that $w$ and $w'$ are tail equivalent (or cofinal) if they are $n$-tail equivalent for some $n$.
The $n$-tail equivalence class of $w$ is denoted by $T_n(w)$ and $T(w) = \bigcup_{n \in \N} T_n(w)$ is the cofinality class of $w$.
\begin{lemma}
  Let $g \in \autB(X^*)$. There are only finitely many $w$ such that $(g,w)$ is not in $\T$.
  If $(g,w)$ is in $\T$, then $w$ and $g(w)$ are cofinal.
  \label{lem:autbcofinal}
\end{lemma}
\begin{proof}
  The $w$ where the germ of $g$ is not in $\T$ are those where the sections along all prefixes are nontrivial. So they can be identified with the projective limit
  $\varprojlim \left\{ v \in X^n \colon g_{|v} \not= \Id \right\}$. Since $g \in \autB(X^*)$, the sets in the limit are uniformly bounded. Hence the projective limit is also finite.
  This proves the first claim. For the second claim, if $(g,w)$ is in $\T$ then $w$ factors as $vu$ with $g_{|v}$ trivial, so $g(w) = g(vu) = g(v)g_{|v}(u) = g(v)u$, so $w$ and $g(w)$ are cofinal.
\end{proof}

\subsection{Bounded Automata}
\begin{definition}
  An automorphism $g \in \Aut(X^*)$ is called a \emph{finite state} automorphism
  if the set of sections $\{ g_{|v} \colon v \in X^* \}$ is finite.

  We denote by $\autBfs(X^*;P)$ the subgroups of finite state automorphisms in $\autB(X^*;P)$. Note that
  every $g \in \autF(X^*)$ is a finite state automorphism.

  An automorphism $g \in \aut(X^*)$ is called \emph{directed} if there is a word $v \in X^n$ with
  $g_{|v} = g$ and $g_{|u} \in \autF(X^*)$ for all $u \in X^n, u \not= v$.
\end{definition}

Every finitary automorphism is a finite state automorphism. A directed automorphism has bounded activity growth.
We will use the following structural result about finite state automata of bounded activity growth, see~\cite{sidki2000}.

\begin{lemma}
  Let $g \in \autB(X^*)$ be a finite state automorphism. Then there exists a $n$ such that for all $v \in X^n$,
  $g_{|v}$ is either directed or finitary.
  \label{lem:structbounded}
  \qedcited
\end{lemma}
\section{Random walks}
\subsection{Potential theoretic background}
We will use the potential theoretic setting as in~\cite{woess2000random}:

Let $\Net = (X,E,r)$ be a network, i.e.\ $(X,E)$ is a connected locally finite graph, and
$r \colon E \rightarrow (0,\infty)$ is a function. We think of $r(e)$ as the resistance
of $e$ and denote by $a(e) = 1/r(e)$ the conductivity of $e$. If $Y$ is a subset of $X$, we denote by $\charr_Y \colon X \rightarrow \{0,1\}$ the characteristic function of $Y$.

Our main examples
will be Schreier graphs: if $G$ is a group generated by a finite set $S$ and $G$
has a left action on $X$, then $\Gamma(G,S,X)$ is the graph with vertex set $X$
and edges $x \rightarrow s(x)$ for every $x \in X, s \in S$, all of unit
resistance. We allow parallel edges and loops.

We are mostly interested in the space $\D(\Net)$ of functions $f \colon X
\rightarrow \R$ with finite
Dirichlet energy $D(f) = \sum_{e \in E} a(e){\left(f(e^+)-f(e^-)\right)}^2$. For any choice
of base point $o$, $\D(\Net)$ is a Hilbert space with norm $\norm{f}^2_{D,o} =
D(f) + \norm{f(o)}^2$. All choices of $o$ give equivalent norms, so there is a
well-defined topology on $\D(\Net)$, so that $f_n$ converges to $f$ if and only
if $\lim_n D(f_n - f) = 0$ and $f_n$ converges to $f$ point-wise.

Let $\D_0(\Net)$ be the closure of functions with finite support in $\D(\Net)$.
By~\cite[Theorem I.2.12]{woess2000random}, the random walk on $\Net$ is recurrent if and
only if $\charr_X \in \D_0(\Net)$. We also use $\D_0(\Net)$ to get the following shorting criterion.
\begin{lemma}[{\cite[I.2.19]{woess2000random}}]
  Let $X = \bigcup_{i \in I} X_i$ be a partition of $X$ such that $\charr_{X_i}
  \in \D_0(\Net)$ for all $i \in I$. Consider the \textbf{shorted network} $\Net'$
  with vertex set $I$ and conductivity $a'(i,j) = \sum_{x \in X_i, y \in
    X_j}(a(x,y))$ for $i\not=j$, $a'(i,i) = 0$.
  If $\Net'$ is recurrent then so is $\Net$.
  \label{lem:shorten}
  \qedcited
\end{lemma}
As a special case we want to mention the Nash-Williams criterion~\cite{nashwilliams}:
\begin{lemma}[{\cite[I.2.20]{woess2000random}}]
  Let $X_1 \subset X_2 \subset \dots$ be an increasing chain of subsets of $X$ with $X_i \in \D_0(\Net)$,
  such that $\partial X_i \subset X_{i+1}$, and $\bigcup X_i = X$. Let $a'_i \coloneqq \sum_{x \in X_i, y \in X \setminus X_i} a(x,y)$.
  If $\sum \frac{1}{a'_i} = \infty$ then $\Net$ is recurrent.
  \label{lem:nashwilliams}
  \qedcited
\end{lemma}
We will also use the following lemma.
\begin{lemma}
  Let $\Net = (X,E,r)$ be a network, $Y \subset X$ with $\partial Y$ finite. Suppose $\Net' = (Y, E', r')$ is a network
  on $Y$ obtained from $\Net$ by restricting to $Y$ and adding and removing finitely many edges and changing finitely many resistances.

  Suppose $\Net'$ is a recurrent network. Then $\charr_Y$  is in  $\D_0(\Net)$.
  \label{lem:technical}
\end{lemma}
\begin{proof}
  Since $\Net'$ is recurrent, $\charr_Y$ is in $\D_0(\Net')$. So there is a
  sequence $f_n \colon Y \rightarrow \R$ such that $\lim_n D_{\Net'}(f_n -
  \charr_Y) = 0$ and $f_n \rightarrow 1$ point-wise on $Y$.

  We extend $f_n$ to $X$ by $0$.  Then $D_{\Net'}(f_n -
  \charr_Y)$ and $D_{\Net}(f_n - \charr_Y)$ differ in only finitely many
  summands, and these go to 0 by point-wise convergence of the $f_n$. So we have
  $\lim_n D_{\Net}(f_n - \charr_Y) = 0$ and thus $\charr_Y \in \D_0(\Net)$.
\end{proof}

\subsection{Recurrence on orbital Schreier graphs}

\begin{definition}
  Let $A$ be a group, $L$ a left $A$-set. We say that the action of $A$ on $L$ is recurrent if for all finitely supported symmetric measures $\lambda$ on $A$,
  the random walk on $L$ induced by $\lambda$ is recurrent for all starting points $l_0 \in L$.
  \label{def:recurrent}
\end{definition}
\begin{remark}
	If $A$ is finitely generated, it is enough to show this for one finitely supported symmetric measure whose support generates $A$. If $S$ is a finite generating set of $A$,
	it is enough to consider the simple random walk on the Schreier graph $\Gamma(G,S,X)$. See for example~\cite{woess2000random}. With this definition it is also clear that 
	recurrent actions are closed under taking subgroups.
	\label{rem:recurrentsubgroups}
\end{remark}

\begin{lemma}
  Let $A, B$ are groups, $L$ a left $A$-set, $M$ a left
  $B$-set such that the actions are both recurrent. Then the action of
  $\wwr{B}{A}{L}$ on $L \times M$ is also recurrent.
  \label{lem:recurrentwr}
\end{lemma}
\begin{proof}
  Let us first reduce to the case where $A$ and $B$ are both finitely generated and both actions
  are transitive:

  Let $(l,m) \in L \times M, \lambda$ a symmetric finitely supported measure on $\wwr{B}{A}{L}$.
  Then there are finitely generated subgroups $A' \subset A, B' \subset B$ such that $\supp(\lambda) \subset \wwr{B'}{A'}{L}$.
  So wlog.\ let $A$ and $B$ be finitely generated. Let $L'$ be the orbit of $l$.
  Then we have a quotient map $\pi \colon \wwr{B}{A}{L} \rightarrow \wwr{B}{A}{L'}$ and we can replace $\lambda$ by
  $\pi_*(\lambda)$ to assume wlog.~that the action of $A$ on $L$ is transitive. We can easily replace $M$ with the orbit of $m$.

  We can now assume that $S$ and $T$ are finite generating sets of $A$ and $B$
  respectively, and both actions are transitive.  Instead of showing recurrence
  for arbitrary $\lambda$, we can now fix a preferred generating set of
  $\wwr{B}{A}{L}$ and show recurrence of the simple random walk on the Schreier
  graph.

  Fix any base point $l_0 \in L$. We take as our generating set of
  $\wwr{B}{A}{L}$ the set $\iota(S) \cup T@ l_0$, let $\Net$ be the resulting network on the Schreier graph. 
  We use Nash-Williamson criterion by partitioning $L \times M = \bigcup_{m \in
    M} L \times {m}$. Now $\partial (L \times {m})$ is a finite collection of edges at $(l_0,m)$,
    and the random walk on $\Gamma(A,S,L)$ is recurrent. So by Lemma~\ref{lem:technical},
    obtain $\charr_{L \times {m}} \in \D_0(\Net)$.
    The shorted network is the Schreier graph of $M$ with respect to $T$, so it is also recurrent. By Lemma~\ref{lem:shorten},
the network $\Net$ is also recurrent.
\end{proof}

\begin{lemma}
  Let $G$ be a finitely generated subgroup of $\autB(X^*)$. Assume that the
  action of $G$ on every finite level is recurrent. Then the
  action of $G$ on every component of the orbital Schreier graph is recurrent.
  \label{lem:recurrentautb}
\end{lemma}
\begin{proof}
  Let $S$ be a finite symmetric generating set of $G$.
  Let $K > 0$ be a uniform bound on $\alpha_n(s)$ for all $n\in \N, s \in S$.
  Let $\Omega$ be a component of the orbital Schreier graph. Let $\Net$ the 
  network on associated with the simple random walk on $\Omega$.

  Let $E$ be the set of edges in $\Net$ which go between different cofinality classes. 
  By Lemma~\ref{lem:autbcofinal}, $E$ is finite. Since $\Omega$ is connected, its vertex set
  must by contained in finitely many cofinality classes $C_1, \dots C_n$. Choose
  representatives $w_i \in C_i \cap \Omega$.

  We claim that $\partial T_m(w_i)$ is uniformly bounded by $K\card{S}$: in fact, if $u$ is the $m$-tail of $w_i$,
  then $\partial T_m(w_i)$ can be identified with the set $\left\{ (s, v) \in S \times X^m \colon s_{|v}(u) \not= u \right\}$.
  This set is contained in $\left\{ (s, v) \in S \times X^m \colon s_{|v} \not= \Id \right\}$, so the bound is clear.

  Since $\Omega$ is connected, $\partial(T_m(w_i) \cap \Omega)$ is also
  uniformly bounded by $K\card{S}$ and $T_m(w_i) \cap \Omega$ has only finitely
  many components. Each such component is a subnetwork of the (recurrent)
  random walk of $G$ on level $m$, so by Lemma~\ref{lem:technical}, their characteristic functions are in $\D_0(\Net)$.

  Let $X_m \coloneqq \bigcup_{1\leq i \leq n} T_m(i) \cap \Omega$. Then $\charr_{X_m}$ is the finite sum of
  characteristic functions of components of $T_m(w_i) \cap \Omega$, so we obtain $\charr_{X_m} \in \D_0(\Net)$.
  Also, $\partial X_m \subset \bigcup_{1 \leq i \leq n} \partial(T_m(w_i) \cap \Omega)$, so
  $\partial X_m$ is uniformly bounded by $nK \card{S}$.

  We can now take a subsequence $X_{m_i}$ such that $\partial X_{m_i}$ is properly contained in $X_{m_{i+1}}$.
  By applying Lemma~\ref{lem:nashwilliams} to the sequence $X_{m_i}$, the random walk on $\Net$ is recurrent.
\end{proof}

\section{Amenability of groups generated by bounded activity automata}
In this section we will prove the following theorem:
\setcounter{theorem}{1}
\begin{theorem}
  Let $P$ be an amenable subgroup of $\Sym(X)$. Suppose that the action of $P$ on $X$ is recurrent.
  Then $\autBfs(X^*;P)$ is amenable.
  \label{thm:maintheorem}
\end{theorem}
We will use the following criterion:
\begin{theorem}[Theorem~3.1 in~\cite{juschenko2016}]
  Let $G$ be a finitely generated group of homeomorphisms of a topological space $Y$, and $\G$
  be its groupoid of germs. Let $\HH$ be a groupoid of germs of homeomorphisms of $Y$. Suppose that the following conditions hold:
  \begin{enumerate}
    \item The group $\left[\left[ \HH \right]\right] \cap G$ is amenable. 
    \item For every $g \in G$ the germ of $g$ at $y$ belongs to $\HH$ for all but finitely many $y \in Y$. We say that $y \in Y$ is \emph{singular}
        if there exists $g \in G$ such that $(g,y) \notin \HH$.
    \item For every singular point $y \in Y$ the orbital Schreier graph $\Gamma(y,G)$ is recurrent.
    \item The isotropy groups $\G_y$ are amenable.
  \end{enumerate}
  Then the group $G$ is amenable.
  \label{thm:JNdlS}
  \qedcited
\end{theorem}
\begin{remark}
  This is almost Theorem~3.1 in~\cite{juschenko2016}, but we weakened the condition (1) from $ \left[\left[ \HH \right]\right]$ amenable to $ \left[\left[ H \right]\right]\cap G$ amenable.
  The original proof only used the weaker condition.
\end{remark}

\begin{proof}[Proof of Theorem~\ref{thm:maintheorem}]
  In order to show amenability of $\autBfs(X^*;P)$, it is enough to show amenability of every finitely generated subgroup of $\autBfs(X^*;P)$.
  So let $G =\langle S\rangle$ be a finitely generated subgroup of $\autBfs(X^*;P)$. 
  We will use Theorem~\ref{thm:JNdlS} with $G$ acting on $X^\omega$, and $\HH = \T$.
  We will show that each condition of Theorem~\ref{thm:JNdlS} is satisfied.
  
  \begin{enumerate}
    \item The group $\left[\left[ \T \right]\right] \cap G$ is amenable:
      In fact $\left[\left[ \T \right]\right] \cap \autBfs(X^*;P) = \autF(X^*;P)$. This  
      follows easily from Lemma~\ref{lem:structbounded}. Now $\autF(X^*;P)$ is the direct limit of iterated wreath products of $P$, so it is amenable,
      hence $\left[\left[ \T \right]\right] \cap G$ is amenable.
      
    \item This follows directly from Lemma~\ref{lem:autbcofinal}.
    \item By inductive application of Lemma~\ref{lem:recurrentwr},
      we see that the action of $n$-th iterated wreath product of $P$ on $X^n$
      is recurrent.  Hence by Remark~\ref{rem:recurrentsubgroups}, the action of
      $G$ on every level is recurrent. Since $G \subset \autBfs(X^*;P) \subset
      \autB(X^*)$, we get by Lemma~\ref{lem:recurrentautb} that $G$ acts
      recurrently on all orbital Schreier graphs.

    \item We encapsulate the proof in the following lemma.
  \end{enumerate}
\end{proof}
\begin{lemma}
  Let $P$ be an amenable subgroup of $\Sym(X)$. Let $G$ be a finitely generated subgroup of $\autBfs(X^*;P)$, $w \in X^\omega$. Then the isotopy group $\G_w$ is amenable.
  \label{lem:isotopygroups}
\end{lemma}
\begin{proof}
  By replacing $X$ with $X^N$ and $P$ with $P_N$ and possibly enlarging the group $G$ itself, we can use Lemma~\ref{lem:structbounded} to assume wlog.\ the following:
  \begin{itemize}
    \item $G$ has a symmetric self-similar generating set $S$.
    \item For all $s \in S$ and $x \in X$, the section $s_{|x}$ is either finitary or directed.
    \item For every directed $s \in S$, there is a $x \in X$ with $s_{|x} = s$. So every directed generator is directed along a constant path.
  \end{itemize}
  Let $\Omega = \left\{ w \in X^\omega \colon w \text{ is eventually constant} \right\}$. Then $\Omega$ is invariant under the action of every generator in $S$,
  so $X^\omega \setminus \Omega$ is also invariant under the action of $G$. For every generator, the germs in $X^\omega \setminus \Omega$ are contained in $\T$,
  so for $w \in X^\omega \setminus \Omega$, the isotropy group $G_w$ is contained in $\T_w = \Id$, so it is trivial.

  For a word $w \in \Omega$, and a group element $g \in G$, let $v_n$ be the prefix of $w$ of length $n$ and consider the sequence
  $g_{|v_n}$. We claim that this sequence is eventually constant, and if $w$ is eventually constantly the letter $x$, then for all $y \in X \setminus \left\{ x \right\}$,
  the section $g_{|v_{n}y}$ is contained in $\autF(X^*;P)$ for $n$ large enough.

  This is true for the generating set by direct inspection and the statement follows by induction over the word length of $g$.

  In particular, for $w \in \Omega$ eventually constantly $x \in X$, we get a group homomorphism 

  \begin{eqnarray*}
    G_w &\rightarrow& \wwr{\autF(X^*;P)}{P_x}{X\setminus\left\{ x \right\}} \\
    g &\mapsto& (y \mapsto g_{|v_{n}y}, \rho_1(g_{|v_{n}})) \text{ for } n \text{ large enough.}
  \end{eqnarray*}

 Here $P_x$ is the stabilizer of $x\in X$ of the action of $P$ on $X$.
  The group homomorphism is injective, and the codomain is amenable, so $G_w$ is amenable.
\end{proof}
\section{Outlook}
Our main application of the main theorem are iterated monodromy groups of post-singularly finite entire functions,
see~\cite{Reinkeentire}. We use the version of Theorem~\ref{thm:JNdlS} from~\cite{juschenko2016}, which impose a recurrence condition
on the random walk on the orbital Schreier graphs. This recurrence condition
was generalized to an extensive amenability condition in~\cite{juschenko2016extensive}. It is shown in~\cite{juschenko2016extensive} that every recurrent action is also extensive amenable. In our Theorem~\ref{thm:maintheorem}, it would be interesting to see whether we could weaken the recurrence condition to an condition about extensive amenability.
Another direction to generalize is to step up in the hierarchy of automata with polynomial activity growth.
In~\cite{AAV, juschenko2016}, it is shown that
the group of automata of linear activity growth acting on a finite alphabet is amenable. Again, crucial step here is
the recurrence of the random walk of the orbital Schreier graphs. It is not clear how this generalizes to infinite alphabets, as it seems that the estimates to show recurrence used finiteness of the alphabet at an important point.

\end{document}